\documentclass[12pt]{article}
\usepackage{latexsym,amssymb,amsmath,a4wide,theorem}
\usepackage[isolatin]{inputenc}
\usepackage[T1]{fontenc}
\usepackage{times}
\usepackage{cite}
\usepackage{amsfonts}

\def\sqr#1#2{{\vbox{\hrule height.#2pt
     \hbox{\vrule width.#2pt height#1pt \kern#1pt
           \vrule width.#2pt}
     \hrule height.#2pt}}}

\def\proof{{\bf Proof:}\enspace}

\newcommand{\dis}{\displaystyle}

\newtheorem{theorem}{Theorem}[section]
\newtheorem{lemma}[theorem]{Lemma}
\newtheorem{remark}[theorem]{Remark}

\newtheorem{proposition}[theorem]{Proposition}

\newcommand{\rz}{{\Bbb R}}

\newcommand{\N}{{\mathcal N}}

\newcommand{\pa}{\partial}
\newcommand{\Om}{\Omega}

\newcommand{\la}{\lambda}
\newcommand{\La}{\Lambda}

\newcommand{\na}{\nabla}
\newcommand{\De}{\Delta}
\newcommand{\de}{\delta}

\newcommand{\nn}{\nonumber}

\newcommand{\beam}{\begin{eqnarray}}
\newcommand{\eeam}{\end{eqnarray}}
\newcommand{\beao}{\begin{eqnarray*}}
\newcommand{\eeao}{\end{eqnarray*}}
\newcommand{\barr}{\begin{array}}
\newcommand{\earr}{\end{array}}
\newcommand{\beqq}{\begin{equation}}
\newcommand{\eeqq}{\end{equation}}

\title{Solutions for biharmonic equations with steep potential wells}

\author{ {\bf  Yuxia Guo}$^a$\,\quad  {\bf Zhongwei Tang}$^b$\thanks{The research was supported by National Science Foundation of China(11571040)}\quad {\bf Lushun Wang}$^{b}$\\
\small $^a$ Department of Mathematics, Tsinghua University, Beijing, 100084, P.R.China\\
\small $^b$ School of Mathematical Sciences, Beijing Normal University, Beijing 100875, P.R.China
}
\date{}

\begin{document}
\maketitle

\begin{abstract}
In this paper, we are concerned with the existence
of least energy solutions for the following biharmonic
equations: $$\Delta^2 u+(\lambda
V(x)-\delta)u=|u|^{p-2}u\quad \hbox{ in } \mathbb{R}^N,\eqno{(P)}$$
where $N\geq 5, 2<p\leq\frac{2N}{N-4}, \lambda>0$ is a parameter, $V(x)$ is a nonnegative potential function with nonempty zero sets  $\mbox{int} V^{-1}(0)$, $0<\delta<\mu_0$ and $\mu_0$ is the principle eigenvalue of  $\Delta^2$ in the zero sets $\mbox{int} V^{-1}(0)$ of $V(x)$. Here $\mbox{int} V^{-1}(0)$ denotes the interior part of the set $V^{-1}(0):=\{x\in\rz^N: V(x)=0\}$. We prove that equation $(P)$ admits a
least energy solution which is trapped  near the zero sets $\mbox{int} V^{-1}(0)$  for $\la>0$ large.
\end{abstract}


{\bf AMS} Subject Classification: 35Q55,  35J655

%\tableofcontents
\numberwithin{equation}{section}
\section{Introduction and main results}
We consider the following biharmonic equation:
\begin{equation}\label{maineq}
\begin{cases}\Delta^2 u+(\lambda V(x)-\delta)u=|u|^{p-2}u
 \quad\hbox{in}\;\mathbb{R}^N,\\ u\in H^2(\rz^N),
\end{cases}
\end{equation}
where $N\geq 5$, $\lambda>0$ is a parameter, $2< p\leq2^{**}$, $2^{**}:=\frac{2N}{N-4}$ is the critical Sobolev exponent for biharmonic operator.

In last decades, biharmonic equation or even its higher version of polyharmonic equation has gotten great attention due to its application in physic and geometry.  In fact, as a mathematical modeling, biharmonic equation can be used to describe some phenomenas appeared in physics and engineering,  such as, the problem of nonlinear oscillation in a suspension bridge (see Lazer and McKenna \cite{lazer}, McKenna \cite{walter}) and the problem of the static deflection of an elastic plate in a fluid (see Abrahams and Davis \cite{davis}). More precisely, when we consider the compatibility equations of elastic mechanics under small deviation of the thin plates, or the Von Karma system describing the mechanic behaviors under large deviation of thin plates, we are forced to study a class of higher order equation or system with biharmonic operator $\Delta^2$.  Mathematically, biharmonic operator is closely related to Paneitz operator, which has been found considerable interest because of its geometry roots.

For the existence and multiplicity of solutions related to biharmonic equations, we firstly refer the reader to the paper by Alves and Miyagaki\cite{alves1}, where they proved the existence of nontrivial solutions to semilinear biharmonic problems with critical nonlinearities.  In  \cite{salvatore}, Salvatore and Squassina proved the existence of infinitely many solutions to a polyharmonic Schrodinger equation with non-homogeneous boundary date on unbounded domain. In \cite{marcos}, Pimenta and Soares  studied the existence and concentration of solutions for a class of biharmonic equations.

In \cite{fp}, using Ljusternik-Schnirelmann theory, Figueiredo and Pimenta proved the existence of multiple solutions of the following biharmonic problem
\[
\begin{cases}\epsilon^4\De^2u+V(x)u=f(u)+\gamma|u|^{2^{**}-2}\quad\mbox{in}\,
\rz^N\\
u\in H^2(\rz^N),
\end{cases}
\]
where $\epsilon>0$ is a small number, $N\geq 5$, $2^{**}=\frac{2N}{N-2}$, $\gamma=0$ or $1$, $V(x)$ is a positive continuous potential and $f(u)$ is subcritical.

Recently, Zhang, Tang and Zhang \cite{ztz} considered the following biharmonic problem
\beqq\label{addeq1}
\begin{cases}\De^2u-\De u+\la V(x)u=f(x,u)\quad\mbox{in}\,\rz^N\\
u\in H^2(\rz^N),
\end{cases}
\eeqq
where $\la>0$, $V\geq0$ is a continuous potential well, $\Om:=\mbox{int}V^{-1}(0)$ is a nonempty bounded domain with smooth boundary, $f(x,u)$ is a function with sublinear growth. For $\la$ large enough, they proved the existence of least energy solutions to (\ref{addeq1}) by taking the infimum of the energy functional over a suitable Banach space. Furthermore, they also proved the solution $u_\la$ of (\ref{addeq1}) which converges along a subsequence in $H^2(\rz^N)$ to a solution of the limit problem
\[
\begin{cases}\De^2u-\De u=f(x,u)\,\,&\mbox{in}\, \Om,\\
\De u=u=0 &\mbox{on}\, \pa\Om,
\end{cases}
\]

More recently, Alves and N\'obrega \cite{an} studied the following problem
\beqq\label{addeq2}
\begin{cases}\De^2u-\De u+(\la V(x)+1)u=f(u)\quad\mbox{in}\,\rz^N,\\
u\in H^2(\rz^N),
\end{cases}
\eeqq
where $\la>0$, $V\geq0$ is a continuous potential well, $\Om:=\mbox{int}V^{-1}(0)$ is a nonempty bounded open set with smooth boundary, $\Om$ has $k$ isolated connected components, $f$ is continuous with subcritical growth. Inspired by Bartsch and Wang \cite{bw}, they established the existence of multi-bump solutions to (\ref{addeq2}) which is trapped near $\Om$ by a deformation Lemma for $\la$ large enough.

For any other related results for biharmonic elliptic equations or polyharmonic equations, we refer the readers
to Deng and Shuai \cite{deng}, Carriao, Demarque and Miyagaki \cite{carriao}, Hu and Wang \cite{hu}, Gazzola and Grunau \cite{grunau}, Guo, Huang and Zhou \cite{guo},  Wang and Shen \cite{wang},  Davila, Dupaigne, Wang and Wei \cite{wei}, Ye and Tang \cite{ye},  and the references therein.
%%%%%%%%%%%%%%%%%%%%%%%

The  aim of this paper is to study the existence and asymptotic behavior of least energy solutions to \eqref{maineq}.

Now we state our assumptions as follows:
\begin{itemize}
  \item[$(V_1)$] $V(x)\in C(\rz^n,\rz)$, $V(x)\geq 0$ and $0<V_\infty:=\liminf_{|x|\to\infty}V(x)<+\infty$;
  \item[$(V_2)$] $\Om:=\mbox{int } V^{-1}(0)$ is a non-empty bounded smooth domain, where $\mbox{int } V^{-1}(0)$ denotes the interior part of the set $V^{-1}(0):=\{x\in\rz^N: V(x)=0\}$;
  \item[$(V_3)$] $0<\delta<\mu_0$, $\mu_0$ is the principle eigenvalue of the operator $\Delta^2$ in $H^2(\Om)\cap H^1_0(\Om)$.
\end{itemize}

\begin{remark}
Indeed, we can replace the condition $(V_1)$ by the following weaker one
\begin{itemize}
  \item[$(\hat{V}_1)$] $V(x)\in C(\rz^n,\rz)$, the set
  $\{x\in\rz^N:0\leq V(x)\leq M_0\}$ is bounded in $\rz^N$ for some $M_0>0$.
\end{itemize}
Precisely, select $M_0=\frac12V_\infty$, according to conditions $(V_1)$ and $(V_2)$, there exists $R>0$ such that
\beqq\label{bV}
\Om\subset\{x\in\rz^N:V(x)\leq M_0\}\subset B_R(0),
\eeqq
where $B_R(0)$ denotes the ball center at $0$ with radii $R$.
\end{remark}

\begin{remark}\label{re:zeroset}
For the smooth assumption of the boundary $\pa\Omega$ in $(V_2)$, we may only assume that $\pa\Omega$ is convex or even weak one: $\pa\Omega$ is Lipschitz continuous and satisfies uniformly outer ball condition. Indeed,  under these assumptions on the boundary, one can check that $\left(\int_\Omega( \Delta u)^2dx\right)^{1/2}$ is equivalent with the standard norm endowed in $H^2(\Omega)\cap H^1_0(\Omega)$ ( see F. Gazzola, H.-Ch. Grunau and G. Sweers \cite[Theorem 2.31]{GGS}).
\end{remark}

Before the statement of our main result,  we introduce some  notations. We set $V_\la(x):=\lambda V(x)-\delta.$ Let $X=:\{u\in
H^2(\mathbb{R}^N)| \int_{\mathbb{R}^N}V(x)u^2dx<+\infty\}$, endowed with the norm:
$$\|u\|_{\la,0}:=\left(\int_{\mathbb{R}^N}(|\De u|^2 +V^+_\lambda(x)u^2 ) dx\right)^{\frac12}$$
where $V^+_\lambda=\max\{V_\lambda,0\}$. It is easy to see that  $(X,\|\cdot\|_{\la,0})$ is a Banach space. For $\la$ large enough, we will prove that
 $$\|u\|_{\la}:=\left(\int_{\mathbb{R}^N}(|\De u|^2 +V_\lambda(x)u^2 )dx\right)^{\frac12}$$
is well defined and indeed a norm which is equivalent to $\|\cdot\|_{\la,0}$ in $X$. For the convenience, we denote the Banach space $(X,\|\cdot\|_\la)$ by $X_\la$.

We define the functional $J_\lambda(u)$ on $X_\la$ by:
\begin{equation}\label{functional}
J_\lambda(u)=\frac{1}{2}\int_{\mathbb{R}^N}(|\De u|^2
 +V_\lambda(x)u^2)dx
 -\frac{1}{p}\int_{\mathbb{R}^N}|u|^{p} dx.
\end{equation}
It is not difficult to verify that the functional $J_\lambda(u)$ is $C^1$ in $X_\la$.

The Frech\'et derivative $J'_\lambda(u)$ is defined by:
\beqq\label{diffphi}
\langle J'_\lambda(u), w\rangle
=\int_{\rz^N}(\De
u\De w+V_\lambda(x)u w)dx -\int_{\rz^N}|u|^{p-2}uw dx, \hbox{ for }w\in X_\lambda.
\eeqq
We say $u$ is a weak solution  of \eqref{maineq} if $u\in X_\la$ such that $J'_\lambda(u)=0$, and $u$ is nontrivial if $u\neq 0$.

We define the Nehari manifold $\N_\lambda$ by
$$\N_\lambda=\{u\in
X_\la\setminus \{0\}: \langle J'_\lambda(u),  u\rangle=0\}$$
and let $$c_\lambda=\inf_{u\in \N_\lambda}J_\lambda(u).$$ We say $u_\la$ is a least energy solution of \eqref{maineq} with least energy $c_\lambda$ if $u_\la\in
\N_\lambda$ such that $c_\lambda$ is achieved.

We also consider the  following  problem in the bounded domain $\Omega=\mbox{int} \{V^{-1}(0)\},$
\begin{equation}\label{limteq}
\begin{cases}\Delta^2
u-\de u=|u|^{p-2}u&\quad \hbox{in}\quad \Omega\\
u\neq 0 &\quad \hbox{in}\quad \Omega\\
 u=0,\Delta u=0, &\quad \hbox{on}\quad\partial\Omega,
\end{cases}
\end{equation}
which is a kind of $limit$ problem of the original problem \eqref{maineq}.
\begin{remark}\label{re:postive}
Normally, because of the lack of maximum principle for biharmonic problems, one can not expect that the least energy solution of \eqref{limteq} is one sign, say positive. In this paper, we only consider nontrivial least energy solutions to \eqref{limteq}.
\end{remark}

Similar as the definitions of the least energy $c_\la$ and the least energy solutions $u_\la$ of \eqref{maineq}, we can also define the least energy $c(\Om)$ and the corresponding least energy solution $u$ of the limit problem \eqref{limteq}.

Our main result in this paper can be stated as follows:
\begin{theorem}\label{th:main-theorem1}
Under the conditions $(V_1)$, $(V_2)$ and $(V_3)$, we assume that $2<p<2^{**}:=\frac{2N}{N-4}$ if $N\geq 5$ or $p=2^{**}$ if $N\geq 8$. Then for $\la$ large enough,  \eqref{maineq} has a least energy
solution $u_{\la}(x)$ which achieves $c_\la$. Moreover,  for any
sequence $\la_n\to \infty$, there exists a subsequence of $\{u_{\la_n}(x)\}$, still denoted by $\{u_{\la_n}(x)\}$, such that $u_{\la_n}(x)$
converges \textmd{in} $H^{2}(\rz^N)$ to a least
energy solution $u(x)$ of \eqref{limteq}.
\end{theorem}

This paper is organized as follows: In Section \ref{sec:pr}, we give some preliminary results. In Section \ref{sec:limit}, we prove the existence of least energy solutions to the limit equation (\ref{limteq}). In Section \ref{sec:els}, we prove the existence of least energy solutions to (\ref{maineq}) for $\la$ large. In Section  \ref{sec:abo}, we study the asymptotic behavior of least energy solutions as $\la\to+\infty$ and give the proof of Theorem \ref{th:main-theorem1}.

\section{Preliminary}\label{sec:pr}

In this section,  we firstly show that the space $(X,\|\cdot\|_{\la,0})$ can be embedded into $H^2(\rz^N)$ uniformly in large $\la$. Secondly, we give some results related to the spectrum of the operator $\Delta^2+\la V-\de$ in the space $(X,\|\cdot\|_{\la,0}).$ Finally, we prove that $\|\cdot\|_{\la,0}$ and $\|\cdot\|_{\la}$ are equivalent norms in $X$ for $\la$ large enough. For simplicity, we denote both the spaces $(X,\|\cdot\|_{\la,0})$ and $(X,\|\cdot\|_{\la})$  by $X_\la$.

%First, we have the following embedding results.
\begin{lemma}\label{le:em}
Assume ($V_1$), ($V_2$) and $(V_3)$ hold, then there exists $\La_0>0$ such that for each $\la>\La_0$ and $u\in X_\la$, we have
\beqq\label{em1}
\|u\|_{H^2(\rz^N)}\leq C\|u\|_{\la,0}
\eeqq
for some $C>0$ which does not depend on $\la$.
\end{lemma}
\proof Let $M_0=\frac12V_\infty$, by (\ref{bV}), we know that
\beqq\label{em2}
V(x)\geq M_0, \forall x\in\rz^N\setminus B_{R}(0) \mbox{ and } \mbox{supp} V^-_{\la}\subset B_R(0), \forall\la>\frac{\de}{M_0},
\eeqq
where $\mbox{supp} V^-_\la$ denotes the support set of $V^-_\la$.

For each $u\in X_\la$ and $\la>\frac{M_0+\de}{M_0}$, it follows from  (\ref{em2}), we have
\beam\label{em3}
\int_{\rz^N\setminus B_R(0)}u^2dx&\leq&\frac{1}{M_0}\int_{\rz^N\setminus B_R(0)}(\la V(x)-\de)u^2dx\nn\\
&\leq&\frac{1}{M_0}\int_{\rz^N\setminus B_R(0)}V_\la^+ u^2dx\nn\\
&\leq&\frac{1}{M_0}\int_{\rz^N}(|\De u|^2+V_\la^+ u^2)dx.
\eeam
By H\"older's inequality and Sobolev inequality, we obtain
\beam\label{em4}
\int_{B_R(0)}u^2dx&\leq&\left(\int_{B_R}|u|^{\frac{2N}{N-4}}dx\right)^{\frac{N-4}{N}}|B_R|^{\frac4N}\nn\\
&\leq& C_1|B_R|^{\frac{4}{N}}\int_{\rz^N}|\De u|^2dx\nn\\
&\leq& C_1|B_R|^{\frac{4}{N}}\int_{\rz^N}(|\De u|^2+V_\la^+ u^2)dx.
\eeam
Combining (\ref{em3}) and (\ref{em4}), we have
\[
\int_{\rz^N}(|\De u|^2+u^2)dx\leq\left(\frac{M_0+1}{M_0}+C_1|B_R|^{\frac{4}{N}}\right)\int_{\rz^N}(|\De u|^2+V_\la^+ u^2)dx.
\]
Thus (\ref{em1}) holds for $\La_0=\frac{M_0+\de}{M_0}$ and $C=\sqrt{\frac{M_0+1}{M_0}+C_1|B_R|^{\frac{4}{N}}}$.\hfill{$\Box$}
\vskip8pt

Let $L_0=\De^2-\de$, $L_\la=\De^2+V_\la$ and $\sigma_{ess}(L_\la)$ be the essential spectrum of $L_\la$ in $X_\la$. In the following, we are going to discuss some results related to the spectrum of the operator $L_0$ and $L_\lambda.$

\begin{lemma}\label{le:ess}
Under the conditions ($V_1$), ($V_2$) and ($V_3$), for each $\la>\La_0$, we have
\[
\sigma_{ess}(L_\la)\subset [\la M_0-\de, +\infty).
\]
Furthermore, $\inf\sigma_{ess}(L_\la)\to+\infty$ as $\la\to+\infty$.
\end{lemma}
\proof The proof of this lemma is similar to Proposition $2.3$ in \cite{bpw}. For the convenience of readers, we give the sketch of the proof.

Set $W_\la=V_\la-\la M_0+\de=\la(V(x)-M_0)$ and write $W_{\la,1}=\max\{W_\la,0\},\ W_{\la,2}=\min\{W_\la,0\}$. For $\la>\La_0$ and $W_{\lambda,1}\geq 0,$ we have
\beqq\label{eq:le1}
\sigma(\De^2+W_{\la,1}+\la M_0-\de)\subset[\la M_0-\de,+\infty)
\eeqq
 Let $H_\la=\De^2+W_{\la,1}+\la M_0-\de$, then $L_\la=H_\la+W_{\la,2}$.

We claim that $W_{\la,2}$ is a relative form compact perturbation of $L_\la$. Since $W_{\la,2}$ is bounded, then the form domain of $H_\la$ is the same as the form domain $X_\la$ of $L_\la$. Thus we have to show that
\[
X_\la\mapsto X_\la^*: u\mapsto W_{\la,2}\cdot u\mbox{ is compact.}
\]
Select a bounded sequence $\{u_n\}_{n\geq 1}$ in $X_\la$, then according to Lemma \ref{le:em}, $\{u_n\}_{n\geq 1}$ is also a bounded sequence in $H^2(\rz^N)$. Thus,  for some $u\in H^2(\rz^N)$, up to a subsequence,
\beqq\label{eq:le2}\left\{
\begin{array}{ll}
u_n\rightharpoonup u\quad\hbox{weakly in}\; H^2(\rz^N),\\
u_n\to u\quad\hbox{strongly in}\; L_{loc}^2(\rz^N),\\
u_n\to u\quad\hbox{a.e. in}\; \rz^N
\end{array}\right.
\eeqq
as $n\to+\infty$. According to (\ref{em2}), we know that $\mbox{supp} W_{\la,2}\subset B_R$ for any $\la>\La_0$. Hence by H\"older's inequality, Sobolev inequality and Lemma \ref{le:em}, for any $\la>\La_0$, $v\in X_\la$, we have
\beam\label{eq:le3}
\left|\int_{\rz^N}W_{\la,2}(u_n-u)vdx\right|&=&\left|\int_{B_R}W_{\la,2}(u_n-u)vdx\right|\nn\\
&\leq&\de\int_{B_R}|(u_n-u)v|dx\nn\\
&\leq&\de\left(\int_{B_R}(u_n-u)^2dx\right)^{\frac12}\left(\int_{B_R}v^2dx\right)^{\frac12}\nn\\
&\leq&\de\left(\int_{B_R}(u_n-u)^2dx\right)^{\frac12}\|v\|_{H^2(\rz^N)}\nn\\
&\leq& C\left(\int_{B_R}(u_n-u)^2dx\right)^{\frac12}\|v\|_{\la,0}.
\eeam
By (\ref{eq:le2}) and (\ref{eq:le3}), we obtain
$$\|W_{\la,2}u_n-W_{\la,2}u\|_{X_\la^*}\leq C\left(\int_{B_R}(u_n-u)^2dx\right)^{\frac12}\to 0,$$
as $n\to+\infty$. Thus $W_{\la,2}$ is a relative form compact perturbation of $L_\la$.

On the other hand, according to the classical Weyl theorem (see Example $3$ in \cite{ReedSimon}, page $117$), we see that $\sigma_{ess}(L_\la)=\sigma_{ess}(H_\la)$. By (\ref{eq:le1}), for $\la>\La_0$, we have $$\sigma_{ess}(L_\la)\subset[\la M_0-\de,+\infty).$$
Moreover, $$\inf\sigma_{ess}(L_\la)\to+\infty\mbox{ as }\la\to+\infty.$$
We complete the  proof of this lemma.\hfill{$\Box$}
\vskip8pt
Now we define:  $$\mu(L_\la)=\inf\left\{\int_{\rz^N}(|\De u|^2+V_\la u^2)dx: u\in X_\la,\int_{\rz^N}u^2dx=1\right\},$$
$$
\mu(L_0)=\inf\left\{\int_{\Om}|\De u|^2-\de u^2dx: u\in H^2(\Om)\cap H^1_0(\Om), \int_{\Om}u^2dx=1\right\}.$$
It is easy to see that $\mu(L_0)$ is the principle eigenvalue of $L_0$ and $\mu(L_0)=\mu_0-\de\geq\mu(L_\la)$. According to Lemma \ref{le:ess}, Theorem XXX.1 in \cite{ReedSimon}, we see that $\mu(L_\la)$ is the principle eigenvalue of $L_\la$ for $\la$ large enough. The following Lemma is related to the limit of $\mu(L_\la)$ as $\la\to+\infty$.

\begin{lemma}\label{le:ape} There exists a $\La_1>\La_0$, such that
for any $\la>\La_1$,  $\mu(L_\la)>\frac{\mu(L_0)}{2}$. Moreover,
$$\mu(L_\la)\to\mu(L_0)\mbox{ as }\la\to+\infty.$$
\end{lemma}
\proof Let $\psi_n\in X_\la$ be the eigenfunction corresponding to $\mu(L_{\la_n})$ such that
\beqq\label{eq:ape1}
\int_{\rz^N}\psi_n^2dx=1, \mbox{ and } \int_{\rz^N}(|\De\psi_n|^2+V_{\la_n}\psi_n^2)dx=\mu(L_{\la_n}).
\eeqq
Then we have
\beao
||\psi_n||_{\la_n,0}^2&=&\int_{\rz^N}(|\De\psi_n|^2+V_{\la_n}\psi_n^2)dx
+\int_{\rz^N}V_{\la_n}^-\psi_n^2dx\\
&=&\mu(L_{\la_n})+\int_{\rz^N}V_{\la_n}^-\psi_n^2dx\leq\mu(L_0)+\de,
\eeao
which implies that  $\{\psi_n\}$ is bounded in $H^2(\rz^N)$. Up to a subsequence, we may assume, for some $\psi\in H^2(\rz^N)$, as $\la_n\to+\infty$,
\beqq\label{eq:ape2}
\left\{
\begin{array}{ll}
\psi_n\rightharpoonup\psi\quad \hbox{weakly in}\; H^2(\rz^N),\\
\psi_n\to\psi\quad \hbox{strongly in}\; L^2_{loc}(\rz^N),\\
\psi_n\to\psi\quad \hbox{a.e. in}\; \rz^N.
\end{array}\right.\eeqq

Firstly, we prove that $\psi\in H^2(\Om)\cap H^1_0(\Om)$. In fact, we just need to verify that $$\psi(x)=0 \mbox{ a.e. in }\rz^N\setminus \Om.$$
For each integer $m\geq1$, we denote $$C_m:=\{x\in \rz^N: V(x)>\frac{1}{m}\}.$$
Now let us fix $m$, and let  $\la_n\to+\infty$, it follows from  (\ref{eq:ape1}) that
\beao
\int_{C_m}\psi^2 dx &\leq& \frac{m}{\la_n}\int_{\rz^N}\la_n V(x)\psi_n^2dx\\
&\leq& \frac{m}{\la_n}\int_{\rz^N}(|\De \psi_n|^2+\la_n V(x)\psi_n^2)dx\\
&\leq&\frac{m}{\la_n}\mu(L_{\la_n})\leq\frac{m}{\la_n}(\mu(L_0)+\de)\to 0.
\eeao
Thus $\psi(x)=0$ a.e. in $C_m$. Notice $\cup_{m=1}^{\infty}C_m=\rz^N\setminus
\Om$, we have $\psi(x)=0$ a.e. in $\rz^N\setminus \Om$.

Secondly, we prove that $\int_{\Omega}\psi^2dx =1.$ In fact, according to (\ref{em2}) and (\ref{eq:ape1}), we have
\beao
\int_{\rz^N\setminus B_R(0)}\psi_n^2dx
&\leq&\frac{1}{M_0\la_n}\int_{\rz^N\setminus B_R(0)}\la_n V(x)\psi_n^2dx\\
&\leq&\frac{1}{M_0\la_n}\int_{\rz^N}(|\De\psi_n|^2+\la_n V(x)\psi_n^2)dx\\
&=&\frac{1}{M_0\la_n}\left(\int_{\rz^N}(|\De\psi_n|^2+ V_{\la_n}\psi_n^2)dx+\de\int_{\rz^N}|\psi_n|^2dx\right)\\
&\leq&\frac{1}{M_0\la_n}(\mu(L_0)+\de)\to 0
\eeao
as $\la_n\to+\infty$. Thus
\beqq\label{eq:ape3}
\lim_{n\to+\infty}\int_{\rz^N\setminus B_R(0)}\psi_n^2dx=0.
\eeqq
Combine (\ref{eq:ape1}), (\ref{eq:ape2}) and (\ref{eq:ape3}), we have
\beao
\int_{\Om}\psi^2dx&=&\lim_{n\to+\infty}\int_{B_R(0)}\psi_n^2dx\\
&=&\lim_{n\to+\infty}\int_{\rz^N}\psi_n^2dx-\lim_{n\to+\infty}\int_{\rz^N\setminus B_R(0)}\psi_n^2dx=1.
\eeao

Finally, we prove $\mu(L_{\la_n})\to\mu(L_0)$ as $n\to+\infty$. In fact, $\psi_n\to\psi$ strongly in $L^2(\rz^N)$ as $n\to+\infty$.
Thus by (\ref{eq:ape1}), we have
\beao
\mu(L_0)&=:&\inf\left\{\int_{\Om}(|\De u|^2-\de u^2)dx: u\in H^2(\Om)\cap H^1_0(\Om),\|u\|_{L^2(\Om)}=1\right\}\\
&\leq&\int_{\Om}(|\De \psi|^2-\de \psi^2)dx\\
&\leq&\lim_{n\to\infty}\int_{\rz^N}(|\De \psi_n|^2+(\la_nV(x)-\de)\psi_n^2)dx\\
&=&\lim_{n\to\infty}\mu(L_{\la_n})\leq \mu(L_0),
\eeao
which implies that $\mu(L_{\la_n})\to \mu(L_0)$ as $n\to\infty$.

Since $\mu(L_\la)$ is increase in $\la$, then $\mu(L_\la)\to \mu(L_0)$ as $\la\to+\infty$. Furthermore, there exists $\La_1>\La_0$ such that for any $\la>\La_1$, we have $\mu(L_\la)>\frac12\mu{(L_0)}>0$.\hfill{$\Box$}\\

Now we prove that $\|\cdot\|_{\la,0}$ and $\|\cdot\|_{\la}$ are equivalent in $X$. Namely,
\begin{lemma}\label{le:en}
For $\la>\La_1$, there exists $C_1>0$ and $C_2>0$ such that for any $u\in X$, we have
\beqq\label{en}
C_1\|u\|_{\la,0}\leq\|u\|_{\la}\leq C_2\|u\|_{\la,0}
\eeqq
where both of $C_1$ and $C_2$ are independent of $\la$.
\end{lemma}
\begin{proof}
For $\la>\La_1$ and any $u\in X_\la$, we have
\[
\int_{\rz^N}(|\De u|^2+V_\la u^2)dx\geq\mu(L_\la)\int_{\rz^N}u^2dx\geq
\frac{\mu(L_0)}{2}\int_{\rz^N}u^2dx.
\]
Then
\[\barr{ll}
\displaystyle\int_{\rz^N}(|\De u|^2+V^+_\la u^2)dx&=\displaystyle\int_{\rz^N}(|\De u|^2+V_\la u^2)dx+
\int_{\rz^N}V^-_\la u^2dx\vspace{0.2cm}\\
&=\displaystyle\int_{\rz^N}(|\De u|^2+V_\la u^2)dx+\de\int_{\rz^N}u^2dx\vspace{0.2cm}\\
&\leq\displaystyle\frac{\mu(L_0)+2\de}{\mu(L_0)}\int_{\rz^N}(|\De u|^2+V_\la u^2)dx.\earr
\]
It is easy to see that
\[
\int_{\rz^N}(|\De u|^2+V_\la u^2)dx\leq\int_{\rz^N}(|\De u|^2+V^+_\la u^2)dx.
\]
Thus (\ref{en}) holds if we select $C_1=\sqrt{\frac{\mu(L_0)}{\mu(L_0)+2\de}} $ and $C_2=1$.
\end{proof}

\begin{remark}
In the following sections, without especially stated, $X_\la$ denotes the space $(X,\|\cdot\|_\la)$. According to the above lemmas, we know that for $\la>\La_1$, $X_\la$ can be continuously imbedded into $H^2(\rz^N)$ uniformly in $\la$. Moreover, $X_\la$ can be continuously imbedded into $L^p(\rz^N)$ for $2\leq p\leq 2^{**}$ and can be compactly imbedded into $L_{loc}^p(\rz^N)$ for $2<p<2^{**}$. All these embedding constants are independent of $\la$.
\end{remark}
\section{Limit problem}\label{sec:limit}
In this section, we consider the limit problem defined in  $\Omega:=V^{-1}(0)$ as follows:
\begin{equation}\label{lmt}
\begin{cases} \Delta^2 u-\de u=|u|^{p-2}u, &\quad x\in \Omega,\\
u=0,\Delta u=0, &\quad x\in \partial\Omega.
\end{cases}
\end{equation}
We define the corresponding functional  $J_\Omega$ on $H(\Omega):=H^2(\Omega)\cap H_0^1(\Omega)$ by:
\begin{equation}\label{functionalbdd}
J_\Omega(u)=\frac{1}{2}\int_\Omega (|\De u|^2-\de u^2)dx-\frac{1}{p}\int_\Omega |u|^{p}dx.
\end{equation}
And define the Nehari manifold $\N_\Omega$ by
$$\N_\Omega:=\left\{u\in H(\Omega)\setminus\left\{0\right\}:\langle J_\Omega(u), u\rangle=0\right\}.$$

Let $$c(\Omega)=\inf_{\N_\Omega}J_{\Omega}(u).$$

We say  that $u$ is a least energy solution of \eqref{lmt} if $u\in \N_\Omega$ is such that $c(\Omega)$ is achieved. Recall that $\{u_n\}$ is a $(PS)_c$ sequence of $J_\Om$ if $J_\Om(u_n)\to c$ and $J'_\Om(u_n)\to 0$ in $H^*(\Om)$, the dual space of $H(\Om)$, as $n\to +\infty$. $J_\Om$ satisfies the $(PS)_c$ condition if any $(PS)_c$ sequence $\{u_n\}$ contains a convergent subsequence in $H(\Om)$.

\vskip8pt

\begin{remark}\label{re:limit-exis}
By F. Gazzola, H.-Ch. Grunau and G. Sweers \cite[Theorem 2.31]{GGS}, $\|\De \cdot\|_{L^2(\Om)} $is a complete norm  of $H(\Omega)$.
\end{remark}
%By a standard argument,  it is easy to check that $c(\Omega)$ is achieved. Namely we have the following lemma.
\begin{lemma}\label{le:limit1} For $2<p<2^{**}$, $N\geq 5$, then $c(\Omega)$ is achieved by a nontrivial solution $u$ of \eqref{lmt} in $\N_\Omega$.
\end{lemma}
\proof Since the proof is quite standard, for the convenience of the reader, we  give the sketech of the proof.

Indeed, from the definition of $c(\Omega)$ and thanks to Ekeland's Variational Principle, we know that there is a sequence $\{u_n\}\subset \N_\Omega$ such that
\beqq\label{eq:pssequence}
J_\Omega(u_n)\to c(\Omega) \hbox{ and }J'_\Omega(u_n)\to 0 \hbox{ in }H^*(\Omega).
\eeqq
Thus by Remark \ref{re:limit-exis} and the fact that $H(\Omega)\hookrightarrow L^p(\Omega)$ is compact , we immediately obtain that $J_\Omega(u_n)$ satisfies Palais-Smale condition. Namely, \eqref{eq:pssequence} indicates that there is a subsequence of $\{u_n\}$, still denoted by $\{u_n\}$, and $u\in N_\Omega$ such that
$u_n\to u$ in $H(\Omega)$ and
$$
J_\Omega(u)=c(\Omega) ,\quad J'_\Omega(u)=0
$$
which complete the proof of this  lemma.\hfill{$\Box$}
\vskip8pt
Now we focus on the existence of least energy solutions of \eqref{lmt} in critical case. Firstly we have the following estimate for the least energy $c(\Omega)$ when $p=2^{**}.$
\begin{lemma}
Suppose $N\geq 8$, $p=2^{**}$, and $\de>0$, we have
\[c(\Om)<\frac{2}{N}S^{\frac{N}{4}}.\]
\end{lemma}
\proof It is well known that $S$ can be achieved by
\[U_\epsilon=c\left(\frac{\epsilon}{\epsilon^2+|x|^2}\right)^{\frac{N-4}{2}}\]
for each $\epsilon>0$ and $c$ is a constant depend on $N$. We my assume $0\in\Om$. Let $\eta$ be a smooth cut-off function satisfies that
\[
\eta(x)=1 \mbox{ for }x\in B_r(0)\,\,\mbox{and}\,\,\mbox{supp}\eta\subset\Om.
\]
Define $u_\epsilon(x)=\eta(x)U_\epsilon(x)\in H(\Omega)$. By direct calculation, we have
$$
\int_{\Om}|\De u_\epsilon|^2dx=\int_{B_r(0)}|\De U_\epsilon|^2dx+\int_{\rz^N\setminus B_r(0)}|\De u_\epsilon|^2dx=S^{\frac{N}{4}}+O(\epsilon^{N-4}),
$$

$$\int_{\Om}|u_\epsilon|^{2^{**}}dx=\int_{B_r(0)}|U_\epsilon|^{2^{**}}dx+\int_{\rz^N\setminus B_r(0)}|u_\epsilon|^{2^{**}}dx=S^{\frac{N}{4}}+O(\epsilon^N),
$$
and
\beao
\displaystyle\int_{\Om}|u_\epsilon|^{2}dx&=&\displaystyle\int_{B_\epsilon(0)}
|U_\epsilon|^{2}dx+\int_{B_r(0)\setminus B_\epsilon(0)}|U_\epsilon|^{2}dx+\int_{\Om\setminus B_r(0)}|u_\epsilon|^{2}dx\\
&\geq&\displaystyle c^2\int_{B_\epsilon(0)}\left(\frac{\epsilon}{\epsilon^2+\epsilon^2}\right)^{N-4}dx
+c^2\int_{B_r(0)\setminus B_\epsilon(0)}\left(\frac{\epsilon}{|x|^2+|x|^2}\right)^{N-4}\\
&&\displaystyle\quad+c^2\epsilon^{N-4}\int_{\Om\setminus B_r(0)}\eta^2\frac{1}{(\epsilon^2+|x|^2)^{N-4}}dx\\
&\geq&\displaystyle\left\{\barr{ll}d\epsilon^4|\ln{\epsilon}|+O(\epsilon^4),
\vspace{0.2cm}&\mbox{ if }N=8,\\
d\epsilon^4+O(\epsilon^{N-4}), &\mbox{ if }N\geq9.\earr\right.
\eeao
Select $t_\epsilon>0$ such that $t_\epsilon u_\epsilon\in \N_\Om$. Thus
\[
t^2_\epsilon\left(\int_{\Om}(|\De u_\epsilon|^2-\de u^2_\epsilon)dx\right)
=t^{2^{**}}_\epsilon\int_{\Om}|u_\epsilon|^{2^{**}}dx.
\]
This implies
\[t_\epsilon=\left(\frac{\int_{\Om}(|\De u_\epsilon|^2-\de u^2_\epsilon )dx}{\int_{\Om}|u_\epsilon|^{2^{**}}dx}\right)^{\frac{1}{2^{**}-2}}.\]
Therefore, for $\epsilon>0$ small enough, we have
\beao
c(\Om)&\leq&\displaystyle J_\Om(t_\epsilon u_\epsilon)\\
&=&\displaystyle\Big(\frac12-\frac{1}{2^{**}}\Big)t^2_\epsilon\int_{\Om}(|\De u_\epsilon|^2-\de u^2_\epsilon)dx\\
&=&\displaystyle\frac2N\frac{(\int_{\Om}(|\De u_\epsilon|^2-\de u^2_\epsilon )dx)^{\frac{N}{4}}}{(\int_{\Om}|u_\epsilon|^{2^{**}}dx)^{\frac{N-4}{4}}}\\
&<&\frac{2}{N}S^{\frac{N}{4}}.
\eeao

\begin{lemma}\label{le:limit2} For $p=2^{**}$, $N\geq8$, $c(\Omega)$ is achieved by a nontrivial solution $u$ of \eqref{lmt} in $\N_\Omega$.\end{lemma}
\proof: By Ekeland's Variational principle and the definition of $c(\Om)$, we can easily get a $(PS)_{c(\Om)}$ sequence $\{u_n\}$. Moreover, $\{u_n\}$ is bounded in $H(\Omega)$. Then up to a subsequence, we have
\[\left\{\begin{array}{ll}
u_n\rightharpoonup u\quad \hbox{in}\; H(\Omega),\\
u_n\rightharpoonup u\quad \hbox{in}\; L^{2^{**}}(\Om),\\
u_n\to u\quad \hbox{in}\; L^2(\Om).
\end{array}\right.\]
Let $v_n=u_n-u$, by Br\'ezis-Lieb's Lemma, we have
\[\int_{\Om}|\De u_n|^2dx=\int_{\Om}|\De u|^2dx+\int_{\Om}|\De v_n|^2dx+o(1),\] \[\int_{\Om}|u_n|^{2^{**}}dx=\int_{\Om}|u|^{2^{**}}dx+\int_{\Om}|v_n|^{2^{**}}dx+o(1).\]
By direct calculation, we obtain that
\[J_\Om(u_n)=J_\Om(u)+\frac12\int_{\Om}|\De v_n|^2dx-\frac{1}{2^{**}}\int_{\Om}|v_n|^{2^{**}}dx+o(1),\]
and
\[\langle J'_\Om(u_n),u_n\rangle=\langle J'_\Om(u),u\rangle+\int_{\Om}|\De v_n|^2dx-\frac{1}{2^{**}}\int_{\Om}|v_n|^{2^{**}}dx+o(1).\]
It is easy to see that $J'_\Om(u)=0$ and $J_\Om(u)\geq 0$. We may assume that $$b=\lim_{n\to+\infty}\int_{\Om}|\De v_n|^2dx=\lim_{n\to+\infty}\int_{\Om}|v_n|^{2^{**}}dx>0.$$
 On one hand, \[b=\lim_{n\to+\infty}\int_{\Om}|v_n|^{2^{**}}dx=\lim_{n\to+\infty}\int_{\Om}|\De v_n|^2dx\geq S\lim_{n\to+\infty}\left(\int_{\Om}|v_n|^{2^{**}}dx\right)^{\frac{2}{2^{**}}}=Sb^{\frac{2}{2^{**}}},\]
 which implies that $b\geq S^{\frac{N}{4}}$. On  the other hand,\[\frac{2}{N}S^{\frac{N}{4}}>c(\Om)\geq\frac12\lim_{n\to+\infty}\int_{\Om}|\De v_n|^2dx-\frac{1}{2^{**}}\lim_{n\to+\infty}\int_{\Om}|v_n|^{2^{**}}dx=\frac{2}{N}b,\]
which implies  $b<S^{\frac{N}{4}},$ this leads to a contradiction. Therefore, $u_n\to u \mbox{ in } H(\Om)$ and $u$ is an achieved function of $c(\Om)$.\hfill{$\Box$}

\section{Biharmonic equation with potential well}\label{sec:els}
In this section, we study the existence of least energy solutions for \eqref{maineq} both in subcritical and critical cases. In Subsection \ref{sec:Ps}, we present some properties of the $(PS)_c$ sequence of $J_\lambda(u)$. In Subsection \ref{sec:esubc} and  Subsection \ref{sec:ecri}, we prove the existence of least energy solutions to \eqref{maineq} in subcritical case and critical case respectively.

\subsection{Properties of $(PS)_c$ sequence}\label{sec:Ps}
Recall  that $\left\{u_n\right\}\subset X_\lambda$ is called a
$(PS)_c$ sequence for the functional
$J_\lambda(u)$ if  $$J_\lambda(u_n)\to c\mbox{ and }
J'_\lambda(u_n)\to 0\,\,\mbox{in} X_\lambda^*\,\,\mbox{as}\,\,n\to \infty ,$$ where $X_\lambda^*$ is  the dual space of
$X_\lambda.$  We say that the functional
$J_\lambda(u)$ satisfies $(PS)_c$ condition if any of the $(PS)_c$
sequence $\left\{u_n\right\}$, up to a subsequence, converges
strongly in $X_\lambda.$

\begin{lemma}\label{le:psbound} Suppose $2<p<2^{**}$ if $N\geq5$ and $p=2^{**}$ if $N\geq8$. For $\la>\La_1$, if $\left\{u_n\right\}$ is a $(PS)_c$ sequence for $J_\lambda(u),$  then
\beqq\label{eq:un-upperbound}
\lim_{n\to \infty}\|u_n\|^2_\lambda=\frac{2pc}{p-2}.
\eeqq
\end{lemma}
\proof Since $\left\{u_n\right\}$ is a $(PS)_c$ sequence, then for $\la>\La_1$, we have
\beao
&&\displaystyle c+o(1)+o(1)\|u_n\|_\lambda\nn\\
&=&J_\lambda(u_n)-\frac{1}{p}\langle
J'_\lambda(u_n),u_n\rangle\nn\\
&=&\displaystyle\Big(\frac{1}{2}-\frac{1}{p}\Big)\int_{\rz^N}(|\De u_n|^2+V_\lambda u_n^2)dx\nn\\
&=&\dis\frac{p-2}{2p}||u_n||^2_\lambda
\eeao
which immediately implies \eqref{eq:un-upperbound}.\hfill{$\Box$}

\begin{lemma}\label{0isolated}Suppose $2<p<2^{**}$ if $N\geq 5$ and $p=2^{**}$ if $N\geq 8$. For $\la>\La_1$, if $\{u_n\}$ is a $(PS)_c$ sequence for $J_\la(u)$, then one of the following statements holds:
\begin{itemize}
\item[$(i)$]$\liminf\limits_{n\to+\infty}\int_{\rz^N}|u_n|^pdx=0$;
\item[$(ii)$] There exists $\sigma>0$ which is independent of
 $\lambda$ such that
 $$\liminf_{n\to+\infty}\int_{\rz^N}|u_n|^pdx\geq\sigma.$$
\end{itemize}
\end{lemma}
{\bf Proof:} Since $\{u_n\}$ is a $(PS)_c$ sequence of $J_\la$, then for $\la>\La_1$, by Sobolev imbedding theorem, we have
\[
\int_{\rz^N}|u_n|^pdx+o(1)=\int_{\rz^N}
(|\De u_n|^2+V_\la|u_n|^2)dx\geq\La\left(\int_{\rz^N}|u_n|^pdx\right)
^{\frac2p}
\]
where $\La$ is not depend on $\la$.
Thus if
$\liminf\limits_{n\to+\infty}\int_{\rz^N}|u_n|^pdx\neq 0$,
then $$\liminf_{n\to+\infty}\int_{\rz^N}|u_n|^pdx\geq\La^{\frac{p}{p-2}}.$$
We complete the proof of this lemma by selecting $\sigma=\La^{\frac{p}{p-2}}$.
\hfill{$\Box$}

\vskip8pt

\begin{lemma}\label{lem:strong}Suppose $N\geq 5$ and $2<p<2^{**}$. Then for any $\epsilon>0$ there exist $\Lambda_\epsilon>\La_1$ such
that
$$\limsup_{n\to \infty}\int_{B^c_{R}}|u_n|^{p}dx\leq \epsilon$$
where $\{u_n\}$ is a $(PS)_c$ sequence for $J_\la(u)$ with
$\la>\La_{\epsilon}$ and  $c\leq c(\Om)$. Here $B^c_{R}=\{x\in \rz^N:|x|\geq R\}$. Especially, there exists $\La_2>\La_1$ such that
$$\limsup_{n\to \infty}\int_{B^c_{R}}|u_n|^{p}dx\leq \frac{\sigma}{2}.$$
\end{lemma}
{\bf Proof:} For $\la>\La_1$, by (\ref{em2}), Lemma \ref{le:em}, Lemma \ref{le:en} and Lemma \ref{le:psbound}, we have
\[
\begin{array}{ll}
\displaystyle\int_{B^c_R}u_n^2dx &\leq \displaystyle\frac{1}{\lambda
M_0}\int_{B^c_R}(V_\la(x)+\de)u_n^2dx\vspace{0.2cm}\\
&\leq \displaystyle\frac{1}{\lambda
M_0}\int_{\rz^N}(|\De u_n|^2+V^+_\la(x)u_n^2+\de u_n^2)dx\vspace{0.2cm}\\
&\leq \displaystyle\frac{C}{\lambda
M_0}\int_{\rz^N}(|\De u_n|^2+V_\la(x)u_n^2)dx\vspace{0.2cm}\\
&\leq \displaystyle\frac{1}{\lambda M_0}\left(\frac{2pc}{p-2}+o(1)\right)\vspace{0.2cm}\\
&\leq\displaystyle\frac{1}{\lambda
M_0}\left(\frac{2pc(\Om)}{p-2}+o(1)\right)\\
&\to 0\quad \hbox{ as } \lambda\to+\infty.\end{array}
\]
By H\"older's inequality and Sobolev embedding theorem, as $\la\to+\infty$, we have
\[
\begin{array}{ll}
\quad\displaystyle\int_{B_R^c}|u_n|^{p}dx&\leq
C\left(\dis\int_{B_R^c}|u_n|^{2^{**}}
dx\right)^{\frac{(N-4)p\theta}{2N}}\left(\dis\int_{B_R^c}|
u_n|^2 dx\right)^{\frac{p(1-\theta)}{2}}\vspace{0.2cm}\\
&\leq C\displaystyle\|u_n\|_\lambda^{p\theta}\left(\int_{B_R^c}|
u_n|^2dx \right)^{\frac{p(1-\theta)}{2}}\\
&\to 0,
\end{array}
\]
where $\theta=\frac{(p-2)N}{4p}$.
Thus there exists  $\Lambda_\epsilon>\La_1$ such that $\limsup\limits_{n\to+\infty}\int_{B^c_R}u_n^pdx<\epsilon$.\hfill{$\Box$}

\vskip8pt
The following lemma compares $c_\la$ and $c(\Om)$.
\begin{lemma}\label{le:compare}
For $\la>\La_1$, $2<p\leq 2^{**}$, the following estimate holds:
\[0<\sigma\leq c_\la\leq c(\Om).\]
\end{lemma}
{\bf Proof:} For any $\la>\La_1$, $u\in \N_\la$, by Sobolev imbedding theorem, we have
\[
\int_{\rz^N}|u|^p dx=\int_{\rz^N}(|\De u|^2+V_\la u^2)dx\geq\La
\left(\int_{\rz^N}|u|^p dx\right)^{\frac2p}
\]
for some $\La>0$ which is independent of $\la$.
Put $\sigma=\La^{\frac{p}{p-2}}$, we have
$\int_{\rz^N}|u|^p dx\geq\sigma$.
Notice that
\[J_\la(u)=J_\la(u)-\frac12\langle J'_\la(u),u\rangle=\frac{p-2}{2p}\int_{\rz^N}|u|^p dx\geq\frac{p-2}{2p}\sigma>0.\]
Then we obtain that $c_\la\geq\sigma>0$.
Since $\N_\Om\subset\N_\la$, then $c_\la\leq c(\Om)$. Thus we complete the proof of this lemma. \hfill{$\Box$}

\subsection{Existence of least energy solution in subcritical case}\label{sec:esubc}
In this subsection, we are concerned with the existence of least energy solutions for the subcritical case.

\begin{proposition}\label{le:clambdaachieved}
Suppose $N\geq 5$, $2<p<2^{**}$. Then for any $\la>\La_2$,
$c_\lambda:=\inf_{\N_\lambda}J_{\lambda}(u)$ is achieved by some
$u\neq 0$.
\end{proposition}
{\bf Proof:} For any $\la>\La_2$, $2<p<2^{**}$, by the definition of $c_\lambda$ and Ekeland variational principle, there exits a $(PS)_{c_\la}$ sequence $\left\{u_n\right\}$ of
$J_\lambda(u)$. By Lemma \ref{le:psbound}, we know that $\left\{u_n\right\}$ is bounded in $X_\la$. Then up to a
subsequence, we have
$$
\left\{\begin{array}{ll}
u_n\rightharpoonup u\quad\hbox{in}\; X_\la,\\
u_n\rightharpoonup u\quad\hbox{in}\; L^{p}(\mathbb{R}^N),\\
u_n\to u\quad\hbox{in}\;L^p_{loc}(\rz^N),\\
u_n\to u\quad\hbox{a.e. in}\;\rz^N
\end{array}\right.$$
as $n\to \infty$.
Thus $J'_\la(u)=0$ and
$$J_\la(u)=J_\la(u)-\frac12\langle J'_\la(u),u\rangle=\Big(\frac12-\frac1p\Big)\int_{\rz^N}|u|^pdx\geq 0.$$
Let $v_n=u_n-u$, by Br\'ezis Lieb's Lemma, we obtain that
\[
\|u_n\|^2_\la=\|u\|^2_\la+\|v_n\|^2_\la,\quad \|u_n\|^p_{L^p(\rz^N)}=\|u\|^p_{L^p(\rz^N)}+\|v_n\|^p_{L^p(\rz^N)}.
\]
It is easy to obtain that
\[
J_\la(u_n)=J_\la(u)+J_\la(v_n)+o(1),\quad
\langle J'_\la(u_n),u_n\rangle=\langle J'_\la(u),u\rangle+\langle J'_\la(v_n),v_n\rangle+o(1).
\]
According to Lemma $8.1$ and Lemma $8.2$ in \cite{w}, we know that
$v_n$ is a $(PS)_d$ sequence of $J_\la$ with $d=c_\la-J_\la(u)$.
We may assume $\lim_{n\to+\infty}\|v_n\|_{L^p(\rz^N)}^p=b$. If $b=0$, then  $v_n\to 0$ in $X_\la$ which implies $u_n\to u$ in $X_\la$. If $b>0$, then by Lemma \ref{0isolated}, we obtain $b\geq\sigma$. But by Lemma \ref{lem:strong}, we have
$$b=\lim_{n\to+\infty}\|v_n\|_{L^p(\rz^N)}^p
=\lim_{n\to+\infty}\int_{B^c_R}|v_n|^pdx\leq\frac{\sigma}{2}$$
which leads to a contradiction. Thus $u_n\to u$ in $X_\la$ and $J_\la(u)=c_\la>0$. Thus $u\in \N_\la$. Therefore, $c_\la$ is achieved by some $u\in \N_\la$ and $u$ is a nontrivial least energy solution to (\ref{maineq}) for any $\la>\La_2$.\hfill{$\Box$}
\vskip8pt
\subsection{Existence of least energy solution in critical case}\label{sec:ecri}
In this section, we consider the existence of least energy solution for \eqref{maineq} in the critical case $p=2^{**}.$
\begin{proposition}\label{le:criexist}
For $p=2^{**}$, $\la>\La_2$, then $c_\lambda:=\inf_{\N_\lambda}J_{\lambda}(u)$ is achieved by some
$u\neq 0$.
\end{proposition}
\proof According to Lemma \ref{le:psbound}, up to a subsequence, we have
$$\left\{\begin{array}{ll}
u_n\rightharpoonup u\quad\hbox{in}\; H^2(\rz^N),\\
u_n\rightharpoonup u\quad\hbox{in}\; L^{2^{**}}(\rz^N),\\
u_n\to u\quad\hbox{a.e. in}\; \rz^N.
\end{array}\right.
$$
Thus $J'_\la(u)=0$ and
\[
J_\la(u)=J_\la(u)-\frac12\langle J'_\la(u),u\rangle=\Big(\frac12-\frac{1}{2^{**}}\Big)\int_{\rz^N}|u|^{2^{**}}dx\geq0.
\]
Let $v_n=u_n-u$, by Br\'ezis Lieb's lemma, we have
$$
\|u_n\|^2_\la=\|u\|^2_\la+\|v_n\|^2_\la+o(1),
$$
$$
\|u_n\|^{2^{**}}_{L^{2^{**}}(\rz^N)}=\|u\|^{2^{**}}_{L^{2^{**}}(\rz^N)}
+\|v_n\|^{2^{**}}_{L^{2^{**}}(\rz^N)}+o(1).
$$
Moreover, we have
\[J_\la(u_n)=J_\la(u)+J_\la(v_n)+o(1),
\langle J'_\la(u_n),u_n\rangle=\langle J'_\la(u),u\rangle+\langle J'_\la(v_n),v_n\rangle+o(1),
\]
According to Lemma $8.1$ and Lemma $8.2$ in \cite{w}, we know that
$v_n$ is a $(PS)_d$ sequence of $J_\la$ with $d=c_\la-J_\la(u)$.
We may assume that $\lim_{n\to+\infty}\|v_n\|^2_\la=\lim_{n\to+\infty}\|v_n\|^{2^{**}}_{L^{2^{**}}(\rz^N)}=b>0$. On  one hand, we have
\beao
b&=&\lim_{n\to+\infty}\int_{\rz^N}|v_n|^{2^{**}}dx\\
&=&\lim_{n\to+\infty}\int_{\rz^N}(|\De v_n|^2+V^+_\la v_n^2)dx\\
&\geq&\lim_{n\to+\infty}\int_{\rz^N}|\De v_n|^2dx\\
&\geq& S\lim_{n\to+\infty}\left(\int_{\rz^N}|v_n|^{2^{**}}dx\right)^{\frac{2}{2^{**}}}
=Sb^{\frac{2}{2^{**}}},
\eeao
Thus $b\geq S^{\frac{N}{4}}$.

On the other hand, by the definition of $\N_\la$ and $\N_\Omega$ we know that
$\N_\Omega\subset\N_\la$ which implies that
$$c_\la\leq c(\Omega).$$ By Lemma \ref{le:limit2}, we know that
$$0<\sigma\leq c_\la\leq c(\Omega)<\frac2N S^{\frac{N}{4}}.$$
 Thus
\[
\frac2N S^{\frac{N}{4}}>c_\la\geq\lim_{n\to+\infty}\left(\frac12\int_{\rz^N}(|\De v_n|^2+V^+_\la v_n^2)dx-\frac{1}{2^{**}}\int_{\rz^N}v_n^{2^{**}}dx\right)=\left(\frac12-\frac{1}{2^{**}}\right)b,
\]
which implies that $b<S^{\frac{N}{4}}.$ It is a contradiction. Therefore,  $u_n\to u$  strongly in $X_\la$ and $c_\la$ is achieved by $u$ in $\N_\la$. Thus $u\in \N_\la$ is a least energy solution of \eqref{maineq}.\hfill{$\Box$}
\vskip8pt

\section{Asymptotic behavior of least energy solutions}\label{sec:abo}
In this section, we firstly study the asymptotic behavior of $c_\la$ as $\la\to+\infty$. Then we give the proof of our main results.

%We firstly give the asymptotic behavior of $c_\la$ in subcritical case, we have the following lemma.
\begin{lemma}\label{le:clamdalimit}Suppose $c_\lambda=\inf_{\N_\lambda}J_\lambda(u)$,
$2<p<2^{**}$ and $N\geq 5$, then  $$\lim_{\lambda\to +\infty}c_\lambda= c(\Omega).$$
\end{lemma}

\noindent{\bf Proof:} According to Lemma \ref{le:compare}, $c_\lambda\leq c(\Omega)$. Moreover, $c_\la$ is strictly increasing
with respect to $\lambda$. In fact, let $\la>\mu$ and $c_\la$ is achieved by $u\in\N_\la$. Then $J_\la(u)=c_\la$, $u\in\N_\la$. Note that
\[
\int_{\rz^N}(|\na u|+V_\mu u^2)dx<\int_{\rz^N}(|\na u|+V_\la u^2)dx=\int_{\rz^N}|u|^pdx.
\]
Then there exists $0<t<1$ such that $tu\in\N_\mu$. This implies that
\[c_\mu\leq J_\mu(tu)=\Big(\frac12-\frac1p\Big)t^p\int_{\rz^N}|u|^pdx=t^pJ_\la(u)<c_\la.\]
Thus the limit of $c_\la$ exists as $\la\to+\infty$.

Assume that $\lim_{\lambda\to +\infty}c_\lambda=k<c(\Omega)$. Then for any $\lambda_n\to+\infty$ as $n\to +\infty$, we have $c_{\lambda_n}\to k<c(\Omega)$.
We assume that $u_n$ is such that $c_{\lambda_n}$ is achieved, then $\{\|u_n\|_{{\la_n}}\}$ is bounded.
According to Lemma \ref{le:em} and lemma \ref{le:en}, we obtain that   $\left\{u_n\right\}$ is
bounded in $H^2(\rz^N)$. Up to a subsequence, we have
\[\left\{\barr{ll}
u_n\rightharpoonup u\quad\hbox{in}\;H^2(\rz^N),\\
u_n\to u\quad\hbox{in} \; L_{loc}^p(\rz^N),\\
u_n\rightharpoonup u\quad\hbox{in}\; L^p(\rz^N),\\
u_n\to u\quad\hbox{a.e.} \;\hbox{in}\;\rz^N.
\earr\right.\]
Firstly, we claim that $u|_{\Omega^c}=0,$ where $\Omega^c=:\left\{x:x\in
\rz^N\setminus \Omega\right\}.$

If not,  we have $u|_{\Omega^c}\not =0.$
Then there exists a compact subset $F\subset \Omega^c$ with
$\mbox{dist}\left\{F, \partial\Omega\right\}>0$ such that $u|_F\not= 0$ and
$$\int_{F}u_n^2dx\to \int_F u^2dx>0, \hbox{ as } n\to \infty.$$ Moreover, by assumption $(V_2)$, there exists
$\epsilon_0>0$ such that $V(x)\geq \epsilon_0$ for any $x\in F.$
Since
\[
\int_{\rz^N}(|\Delta u_n|^2+ V_{\la_n}(x)u_n^2)dx=\displaystyle\int_{\rz^N}|u_n|^{p}dx,
\]
then
\beao
c_{\la_n}=J_{\lambda_n}(u_n)&=&\displaystyle\frac{1}{2}\int_{\rz^N}(|\Delta
u_n|^2+V_{\la_n}(x)u_n^2)dx-\frac{1}{p}\int_{\rz^N}|u_n|^{p}dx\\
&=&\displaystyle\left(\frac{1}{2}-\frac{1}{p}\right)\int_{\rz^N}(|\Delta
u_n|^2+V_{\la_n}(x)u_n^2)dx\\
&\geq&\displaystyle
\left(\frac{1}{2}-\frac{1}{p}\right)\left(\int_{\rz^N}\la_n V(x)u_n^2dx-\de\|u_n\|^2_{H^2(\rz^N)}\right)\\
&\geq&\displaystyle
\left(\frac{1}{2}-\frac{1}{p}\right)\left(\int_{F}\la_n\epsilon_0u_n^2dx-\de\|u_n\|^2_{H^2(\rz^N)}\right)\\
&\to&+\infty\quad\hbox{as}\;n\to+\infty.\eeao
This contradiction shows that $u|_{\Omega^c}=0$, by the smooth assumption on $\pa\Omega$ we have $u\in H(\Omega)$.

Now we show that
\begin{equation}\label{add1}
u_n\to u\quad\hbox{in}\;L^{p}(\rz^N).\end{equation}
Suppose (\ref{add1}) is not true,  then
by concentration compactness principle of P.Lions (see \cite{lions}), there
exist $\delta>0, \rho>0$ and $x_n\in \rz^N$ with $|x_n|\to
+\infty$ such that
\begin{equation}\label{flp4}
\limsup\limits_{n\to\infty}\int_{B_\rho(x_n)}|u_n-u|^2dx\geq\delta>0.\end{equation}
By the choice of $\left\{u_n\right\}$ and the facts that $u|_{\Omega^c}=0,$  we have
\beao
\displaystyle J_{\lambda_n}(u_n)&\geq&
\dis\left(\frac{1}{2}-\frac{1}{p}\right)\left(\int_{B_\rho(x_n)\cap B^c_R(0)}\la_n V(x)u_n^2dx-\de\|u_n\|^2_{H^2(\rz^N)}\right)\\
&\geq&
\displaystyle\left(\frac{1}{2}-\frac{1}{p}\right)\left(\int_{B_\rho(x_n)\cap B^c_R(0)}\la_n V(x)|u_n-u|^2dx-\de\|u_n\|^2_{H^2(\rz^N)}\right)\\
&\geq& \displaystyle\left(\frac{1}{2}-\frac{1}{p}\right)\left(\la_n M_0\int_{B_\rho(x_n)}|u_n-u|^2dx-\de\|u_n\|^2_{H^2(\rz^N)}\right)\\
&\to& +\infty.\eeao
This contradiction deduces  that $u_n\to u$ in $L^{p}(\rz^N)$.

Since $J'_{\la_n}(u_n)=0$, then for any $\psi\in H(\Om)$, we have
\[
\int_{\rz^N}(\De u_n\De\psi+V_{\la_n}u_n\psi)dx=\int_{\rz^N}|u_n|^{p-2}u_n\psi dx.
\]
Let $n\to+\infty$, we have
\[
\int_{\Om}(\De u\De\psi-\de u\psi)dx=\int_{\Om}|u|^{p-2}u\psi dx.
\]
Thus $J'_\Om(u)=0$.
Since
\[\barr{ll}
J_{\la_n}(u_n)&=\displaystyle J_{\la_n}(u_n)-\frac12\langle J'_{\la_n}, u_n\rangle\vspace{0.2cm}\\
&=\displaystyle \left(\frac12-\frac1p\right)\int_{\rz^N}|u_n|^p dx\vspace{0.2cm}\\
&=\displaystyle\left(\frac12-\frac1p\right)\int_{\Om}|u|^p dx+o(1),\earr
\]
Then by Lemma \ref{le:compare}, $k=(\frac12-\frac1p)\int_{\rz^N}|u|^p dx\geq \sigma>0$ which implies $u\neq 0$. Thus $u\in\N_\Om$ and
$$J_\Om(u)=\left(\frac12-\frac1p\right)\int_{\rz^N}|u|^p dx=k\geq c(\Om).$$
It is a contradiction. The desired result holds true.  Furthermore, $\|u_n-u\|^2_{\la_n}\to 0$ as $n\to+\infty$, then according to Lemma \ref{le:em} and Lemma $\ref{le:en}$, we have  that $u_n\to u$ in $H^2(\rz^N)$.
\hfill{$\Box$}
\vskip8pt

%Now we give the asymptotic behavior of $c_\la$ in critical case and which is

\begin{lemma}\label{le:calimit}
Suppose $N\geq 8$ and $p=2^{**}$, then $\lim\limits_{\la\to+\infty}c_\la=c(\Om)$.
\end{lemma}
\proof It is easy to see that $c_\la$ is increase with respect to $\la$ and $c_\la\leq c(\Om)$. Assume $\lim\limits_{\la_n\to+\infty}c_{\la_n}=k\leq c(\Om)$. For $n$ large enough, let $u_n\in X_{\la_n}$ satisfies $J_{\la_n}(u_n)=c_{\la_n}$ and $J'_{\la_n}(u_n)=0$. As proved in Lemma \ref{le:psbound}, we can easy to see  that $u_n$ is bounded in $X_{\la_n}$, namely $\|u_n\|_{\la_n}\leq C$ for some $C>0$. And as a result of Lemma \ref{le:em} and Lemma \ref{le:en}, $u_n$ is bounded in $H^2(\rz^N)$. Then up to a subsequence, we have
\[\left\{\barr{ll}
u_n\rightharpoonup u\quad \hbox{in}\; H^2(\rz^N),\\
u_n\rightharpoonup u\quad \hbox{in}\; L^{2^{**}}(\rz^N),\\
u_n\to u\quad \hbox{in }\;L^2_{loc}(\rz^N),\\
u_n\to u\quad \hbox{a.e. in}\; \rz^N.\earr\right.
\]
Similar to the proof of Lemma \ref{le:clamdalimit}, $u=0$ on $\rz^N\setminus\Om$.
And for each $\phi\in H(\Om)$,  we have
\beao
0&=&J'_{\la_n}(u_n)\phi\\
&=&\int_{\rz^n}(\De u_n\De\phi+V_{\la_n}u_n\phi )dx-\int_{\rz^N}u_n^{2^{**}}\phi dx\\
&\to&\int_{\Om}(\De u\De\phi-\de u^2\phi)dx-\int_{\Om}u^{2^{**}}\phi dx\\
&=&J'_\Om(u)\phi, \hbox{ as } n\to+\infty,
\eeao
Thus $J'_\Om(u)=0$. Furthermore,
\[
J_\Om(u)=J_\Om(u)-\frac12\langle J'_\Om(u),u\rangle
=\left(\frac12-\frac{1}{2^{**}}\right)\int_{\Om}|u|^{2^{**}}dx\geq 0.
\]
Let $v_n=u_n-u$, by Br\'ezis-Lieb's Lemma, we have
$$
\int_{\rz^N}|\De u_n|^2dx=\int_{\Om}|\De u|^2dx+\int_{\rz^N}|\De v_n|^2dx+o(1),
$$
$$
\int_{\rz^N}|u_n|^{2^{**}}dx=\int_{\Om}|u|^{2^{**}}dx+\int_{\rz^N}|v_n|^{2^{**}}dx+o(1),
$$
and
\beao\int_{\rz^N}V_{\la_n}u^2_ndx&=&\int_{\rz^N}V_{\la_n}u^2dx+\int_{\rz^N}V_{\la_n}v^2_ndx
+\int_{\rz^N}2V_{\la_n}uv_ndx\\
&=&-\de\int_{\Om} u^2dx+\int_{\rz^N}V_{\la_n}v^2_ndx-2\de\int_{\Om}uv_ndx\\
&=&-\de\int_{\Om} u^2dx+\int_{\rz^N}V_{\la_n}v^2_ndx+o(1).\eeao
Thus we obtain
$$
J_{\la_n}(u_n)=J_\Om(u)+J_{\la_n}(v_n)+o(1),
$$
$$
\langle J'_{\la_n}(u_n),u_n\rangle=\langle J'_{\Om}(u),u\rangle+\langle J'_{\la_n}(v_n),v_n\rangle+o(1).
$$
We may assume that
\[b=\lim_{n\to+\infty}\int_{\rz^N}(|\De v_n|^2+V_{\la_n}v_n^2)dx=\lim_{n\to+\infty}\int_{\rz^N}|v_n|^{2^{**}}dx>0.\] On one hand, by Sobolev inequality, we have
\beao
b&=&\lim_{n\to+\infty}\int_{\rz^N}|v_n|^{2^{**}}dx\\
&=&\lim_{n\to+\infty}\int_{\rz^N}(|\De v_n|^2+V_{\la_n}v_n^2)dx\\
&=&\lim_{n\to+\infty}\int_{\rz^N}(|\De v_n|^2+V^+_{\la_n}v_n^2)dx\\
&\geq&\lim_{n\to+\infty}\int_{\rz^N}|\De v_n|^2\\
&\geq&\lim_{n\to+\infty}S(\int_{\rz^N}|v_n|^{2^{**}}dx)^{\frac{2}{2^{**}}}.
\eeao
Thus $b\geq S^{\frac{N}{4}}$.

On the other hand, since
\[
J_{\la_n}(v_n)=J_{\la_n}(v_n)-\frac12\langle J_{\la_n}(v_n),v_n\rangle+o(1)
=\left(\frac12-\frac{1}{2^{**}}\right)\int_{\rz^N}|v_n|^{2^{**}}dx+o(1),
\]
then
\beao\frac{2}{N}S^{\frac{N}{4}}&>&c(\Om)\geq k
\geq\lim_{n\to+\infty}J_{\la_n}(v_n)\\
&=\dis&\left(\frac12-\dis\frac{1}{2^{**}}\right)\lim_{n\to+\infty}\int_{\rz^N}|v_n|^{2^{**}}dx
=\left(\frac12-\frac{1}{2^{**}}\right)b.\eeao
Thus $b<S^{\frac{N}{4}}.$ This leads to a contradiction. Thus we have $u_n\to u$ in $L^{2^{**}}(\rz^N)$. According to Lemma \ref{le:em} and Lemma \ref{le:en}, we see that $u_n\to u$ in $H^2(\rz^N)$. Furthermore, by Lemma \ref{le:compare}, we have
\[\barr{ll}
J_\Om(u)&=\displaystyle \left(\frac12-\frac{1}{2^{**}}\right)\int_{\Om}|u|^{2^{**}}dx\vspace{0.2cm}\\
&=\displaystyle \left (\frac12-\frac{1}{2^{**}}\right)\lim_{n\to+\infty}\int_{\Om}|u_n|^{2^{**}}dx
\vspace{0.2cm}\\
&=\displaystyle \lim_{n\to+\infty}J_{\la_n}(u_n)=k\geq\sigma>0,\earr
\]
then $u\neq 0$. Hence, $u\in\N_\Om$ and
$$c(\Om)\leq J_\Om(u)=k\leq c(\Om)$$ which implies $J_{\Om}(u)=c(\Om)$.\hfill{$\Box$}
\vskip8pt

Now we  are lead to the proof of Theorem \ref{th:main-theorem1}.
\vskip8pt
\noindent{\bf Proof of Theorem \ref{th:main-theorem1}:} The existence of least energy solutions to (\ref{maineq}) is proved by Proposition \ref{le:clambdaachieved} and Proposition \ref{le:criexist} for $\la>\La_2$. The asymptotic behavior of least energy solutions follows from Lemma \ref{le:clamdalimit} and Lemma \ref{le:calimit} for $\la\to +\infty$. Thus we complete the proof of our main result Theorem \ref{th:main-theorem1}.\hfill{$\Box$}
\vskip8pt

%%%%%%%%%%%%%%%%%%%%%%%%%%%%%%%%%%%%%%%%%%%%%%%%%%%%%%%%%%%%%%%
\bibliographystyle{springer}
\bibliography{mrabbrev,literatur}
\newcommand{\noopsort}[1]{} \newcommand{\printfirst}[2]{#1}
\newcommand{\singleletter}[1]{#1} \newcommand{\switchargs}[2]{#2#1}
%%%%%%%%%%%%%%%%%%%%%%%%%%%%%%%%%%%%%%%%%%%%%%%%%%%%%%%%%%%%%%%%%%%%%%%

\end{document}